\documentclass[12pt]{amsart}%
\usepackage{amssymb,amsmath,amsfonts,latexsym,amsthm,geometry}
\usepackage{amsmath}
\usepackage{amsfonts}
\usepackage{amssymb}
\usepackage{graphicx}%
\providecommand{\U}[1]{\protect\rule{.1in}{.1in}}
\providecommand{\U}[1]{\protect\rule{.1in}{.1in}}
\providecommand{\U}[1]{\protect\rule{.1in}{.1in}}
\providecommand{\U}[1]{\protect\rule{.1in}{.1in}}
\newtheorem{theorem}{Theorem}[section]
\newtheorem{corollary}[theorem]{Corollary}
\newtheorem{proposition}[theorem]{Proposition}
\newtheorem{lemma}[theorem]{Lemma}

\theoremstyle{definition}
\newtheorem{problem}[theorem]{Problem}

\newtheorem{definition}[theorem]{Definition}
\begin{document}
\title[Absolutely summing multipolynomials]{Absolutely summing multipolynomials}
\author[T. Velanga]{T. Velanga}
\address{IMECC--UNICAMP, Universidade Estadual de Campinas \\
13.083-859 - S\~{a}o Paulo, Brazil; Departamento de Matem\'{a}tica,
Universidade Federal de Rond\^{o}nia, 76.801-059 - Porto Velho, Brazil.}
\email{thiagovelanga@unir.br, ra115476@ime.unicamp.br}
\thanks{T. Velanga was supported by CAPES Grant 23038.002511/2014-48 and FAPERO Grant 01133100023-0000.41/2014}

\begin{abstract}
In this paper, we develop the theory of absolutely summing multipolynomials.
Among other results, we generalize and unify previous works of G. Botelho and
D. Pellegrino concerning absolutely summing polynomials/multilinear mappings
in Banach spaces with unconditional Schauder basis.

\end{abstract}
\maketitle


\section{Introduction}

The basics of the linear theory of absolutely summing operators can be found
in the classical book \cite{diestel}. Its extension to the multilinear setting
was sketched by A. Pietsch in 1983 \cite{pi2} and it was rapidly developed
thereafter in several nonlinear environments.

For the basic theory of homogeneous polynomials and multilinear mappings
between Banach spaces we refer to S. Dineen \cite{dineen} and J. Mujica
\cite{Muj}.

Throughout this paper $X_{1},\ldots,X_{m},X,Y$ will stand for Banach spaces
and the scalar field $\mathbb{K}$ can be either $\mathbb{R}$ or $\mathbb{C}$.
By $X^{\prime}$ we denote the topological dual of $X$ and by $B_{X}$ its
closed unit ball. For each $m\in\mathbb{N}$, $\mathcal{P}\left(
^{m}X;Y\right)  $ ($\mathcal{P}\left(  ^{m}X\right)  $ if $Y=\mathbb{K}$)
denotes the space of continuous $m$-homogeneous polynomials from $X$ into $Y$
endowed with the usual $\sup$ norm; and $\mathcal{L}\left(  X_{1},\ldots
,X_{m};Y\right)  $ ($\mathcal{L}\left(  X_{1},\ldots,X_{m}\right)  $ if
$Y=\mathbb{K}$) denotes the space of continuous $m$-linear mappings from the
cartesian product $X_{1}\times\ldots\times X_{m}$ into $Y$ with the usual
$\sup$ norm.

Let $p\in\lbrack1,\infty)$. The vector space of all sequences $\left(
x_{j}\right)  _{j=1}^{\infty}$ in $X$ such that $\left\Vert \left(
x_{j}\right)  _{j=1}^{\infty}\right\Vert _{p}=\left(
{\textstyle\sum_{j=1}^{\infty}}
\left\Vert x_{j}\right\Vert ^{p}\right)  ^{1/p}<\infty$ will be denoted by
$\ell_{p}\left(  X\right)  $. We will also denote by $\ell_{p}^{w}\left(
X\right)  $ the vector space composed by the sequences $\left(  x_{j}\right)
_{j=1}^{\infty}$ in $X$ such that $\left(  \varphi\left(  x_{j}\right)
\right)  _{j=1}^{\infty}$ in $\ell_{p}\left(  \mathbb{K}\right)  $ for every
continuous linear functional $\varphi:X\rightarrow\mathbb{K}$. The function
$\left\Vert \cdot\right\Vert _{w,p}$ in $\ell_{p}^{w}\left(  X\right)  $
defined by $\left\Vert \left(  x_{j}\right)  _{j=1}^{\infty}\right\Vert
_{w,p}=\sup_{\varphi\in B_{X^{\prime}}}\left(
{\textstyle\sum_{j=1}^{\infty}}
\left\Vert \varphi\left(  x_{j}\right)  \right\Vert ^{p}\right)  ^{1/p}$ is a
norm. The case $p=\infty$ is the case of the bounded sequences and in
$\ell_{\infty}\left(  X\right)  $ we use the $\sup$ norm.

Let us begin by recalling the notions of absolutely summing homogeneous
polynomials and multilinear mappings. These notions dates back to the works of
A. Pietsch \cite{pi2}\ and Alencar--Matos \cite{alenmat}.

\begin{definition}
\label{defabsmultilin}A continuous $m$-homogeneous polynomial $P:X\rightarrow
Y$ is \textit{absolutely} $\left(  p;q\right)  $-\textit{summing} (or $\left(
p;q\right)  $-\textit{summing}) if $\left(  P\left(  x_{j}\right)  \right)
_{j=1}^{\infty}\in\ell_{p}\left(  Y\right)  $ for all $\left(  x_{j}\right)
_{j=1}^{\infty}\in\ell_{q}^{w}\left(  X\right)  $. A continuous $m$-linear
mapping $T:X_{1}\times\cdots\times X_{m}\rightarrow Y$ is \textit{absolutely
}$\left(  p;q_{1},\ldots,q_{m}\right)  $-\textit{summing }(or $\left(
p;q_{1},\ldots,q_{m}\right)  $-\textit{summing}) if $\left(  T\left(
x_{j}^{\left(  1\right)  },\ldots,x_{j}^{\left(  m\right)  }\right)  \right)
_{j=1}^{\infty}\in\ell_{p}\left(  Y\right)  $ for all $\left(  x_{j}^{\left(
k\right)  }\right)  _{j=1}^{\infty}\in\ell_{q_{k}}^{w}\left(  X_{k}\right)  ,$
$k=1,\ldots,m$.
\end{definition}

The space of absolutely $\left(  p;q\right)  $-summing $m$-homogeneous
polynomials from $X$ into $Y$ is denoted by $\mathcal{P}_{\text{\textit{as}%
}\left(  p;q\right)  }\left(  ^{m}X;Y\right)  $ ($\mathcal{P}%
_{\text{\textit{as}}\left(  p;q\right)  }\left(  ^{m}X\right)  $ if
$Y=\mathbb{K}$). Analogously, the space of absolutely $\left(  p;q_{1}%
,\ldots,q_{m}\right)  $-summing $m$-linear mappings from $X_{1}\times
\cdots\times X_{m}$ into $Y$ is denoted by $\mathcal{L}_{\text{\textit{as}%
}\left(  p;q_{1},\ldots,q_{m}\right)  }\left(  X_{1},\ldots,X_{m};Y\right)  $
($\mathcal{L}_{\text{\textit{as}}\left(  p;q_{1},\ldots,q_{m}\right)  }\left(
X_{1},\ldots,X_{m}\right)  $ if $Y=\mathbb{K}$). When $q_{1}=\cdots=q_{m}=q$,
we simply write $\mathcal{L}_{\text{\textit{as}}\left(  p;q\right)  }\left(
X_{1},\ldots,X_{m};Y\right)  $.

A mapping $P:X_{1}\times\cdots\times X_{m}\rightarrow Y$ is said to be an
$\left(  n_{1},\ldots,n_{m}\right)  $\textit{-homogeneous polynomial} if $P$
is an $n_{k}$-homogeneous polynomial in each variable. This space is denoted
by $\mathcal{P}\left(  ^{n_{1}}X_{1},\ldots,^{n_{m}}X_{m};Y\right)  $ and%
\[
\left\Vert P\right\Vert :=\sup\left\{  \left\Vert P\left(  x_{1},\ldots
,x_{m}\right)  \right\Vert ;x_{k}\in X_{k}\text{, }\underset{k}{\max
}\left\Vert x_{k}\right\Vert _{X_{k}}\leq1\right\}  \text{,}%
\]
is a norm in $\mathcal{P}\left(  ^{n_{1}}X_{1},\ldots,^{n_{m}}X_{m};Y\right)
$. When $m=1$ and $n_{1}=1$, it is just the concept of linear operators; when
$m=1$ and $n_{1}>1$, we have the concept of homogeneous polynomials and
finally, when $m>1$ and $n_{1}=\cdots=n_{m}=1$, we recover the concept of
multilinear mappings. Sometimes these particular cases will be called
\textit{extreme cases,} and the intermediary cases: $m>1$ and $n_{j}>1$ for
some $j=1,\ldots,m$, will be called \textit{strict multipolynomial}.

The notion of absolutely summing multipolynomials was introduced in
\cite[Definition 4.5]{Vel}. As it happens with absolutely summing polynomials
and multilinear mappings, absolutely summing multipolynomials can be
characterized by means of sequences (the same happens to multiple summing
multipolynomials). The proofs are similar to the classical proofs; in essence,
the only difference is that we shall use the Banach--Steinhaus Theorem for
multipolynomials. For further reference, we shall denote the space of
absolutely $\left(  q;1,\ldots,1\right)  $-summing $\left(  n_{1},\ldots
,n_{m}\right)  $\textit{-}homogeneous polynomials from $X_{1}\times
\cdots\times X_{m}$ into $Y$ by the short notation $\mathcal{P}%
_{\text{\textit{as}}\left(  q;1\right)  }\left(  ^{n_{1}}X_{1},\ldots,^{n_{m}%
}X_{m};Y\right)  $.

The primary goal of this paper is to start developing the absolutely summing
multipolynomial theory intending to unify the known results concerning
homogeneous polynomials and multilinear mappings which have been done
separately so far. Specifically, we generalize to multipolynomials previous
results of Botelho--Pellegrino \cite{bp} and Pellegrino \cite{pelcot, ppoloni}
concerning absolutely summing polynomials and multilinear mappings which have
been broadening established techniques from the famous paper "Absolutely
summing operators in $\mathcal{L}_{p}$ spaces and their applications" by J.
Lindenstrauss and A. Pe\l czy\'{n}ski \cite{lp}.

Modus vivendi: Every concept separately dealt in the multilinear or polynomial
setting should be designed such that it extends and unifies the whole theory
by the procedure to be tracked in this paper.

\section{Preliminary results}

This section is devoted to displaying principal tools which will be useful
further on. A well-known, albeit unpublished, result due to A. Defant and J.
Voigt states that every scalar-valued $m$-linear mapping is absolutely
$(1;1)$-summing (see \cite[Theorem 3.10]{alenmat}). The polynomial version is
also valid. We start by extending that to multipolynomials.

\begin{lemma}
\label{dvmultipol}Every $\left(  n_{1},...,n_{m}\right)  $-homogeneous
polynomial $P:X_{1}\times\cdots\times X_{m}\rightarrow\mathbb{K}$ is
absolutely $\left(  1;1\right)  $-summing.
\end{lemma}

\begin{proof}
Let $P:c_{0}\times\dots\times c_{0}\rightarrow\mathbb{K}$ be an $\left(
n_{1},\ldots,n_{m}\right)  $-homogeneous polynomial. Define the $M:=\left(
n_{1}+\cdots+n_{m}\right)  $-homogeneous polynomial $Q:c_{0}\rightarrow
\mathbb{K}$ by%
\[
Q(x):=P\left(  \left(  x_{\left(  i-1\right)  m+1}\right)  _{i\in\mathbb{N}%
},\ldots,\left(  x_{\left(  i-1\right)  m+m}\right)  _{i\in\mathbb{N}}\right)
\text{.}%
\]
Note that, since we are dealing with the $\sup$ norm, we have%
\[
\left\Vert Q\right\Vert \leq\left\Vert P\right\Vert
\]
and, since $Q$ is a scalar-valued $M$-homogeneous polynomial, there must be a
constant $C>0$ such that%
\[
\overset{\infty}{\underset{j=1}{%
{\displaystyle\sum}
}}\left\vert P\left(  \left(  x_{\left(  i-1\right)  m+1}^{\left(  j\right)
}\right)  _{i\in\mathbb{N}},\ldots,\left(  x_{\left(  i-1\right)
m+m}^{\left(  j\right)  }\right)  _{i\in\mathbb{N}}\right)  \right\vert
=\overset{\infty}{\underset{j=1}{%
{\displaystyle\sum}
}}\left\vert Q(x^{(j)})\right\vert \leq C\left\Vert P\right\Vert \left\Vert
(x^{(j)})_{j=1}^{\infty}\right\Vert _{w,1}^{M}\text{,}%
\]
whenever $(x^{(j)})_{j=1}^{\infty}\in\ell_{1}^{w}\left(  c_{0}\right)  $. In
particular,%
\[
\overset{\infty}{\underset{j=1}{%
{\displaystyle\sum}
}}\left\vert P(e_{j},\ldots,e_{j})\right\vert =\overset{\infty}{\underset
{j=1}{%
{\displaystyle\sum}
}}\left\vert Q(z^{(j)})\right\vert
\]
where, for each $j\in\mathbb{N}$,%
\[
\left(  z_{i}^{\left(  j\right)  }\right)  _{i\in\mathbb{N}}:=\left\{
\begin{array}
[c]{cc}%
1\text{,} & \text{if }\left(  j-1\right)  m+1\leq i\leq jm\\
0\text{,} & \text{otherwise}%
\end{array}
\right.  \text{.}%
\]
Since%
\[
\left\Vert (z^{(j)})_{j=1}^{\infty}\right\Vert _{w,1}=1\text{,}%
\]
then%
\[
\overset{\infty}{\underset{j=1}{%
{\displaystyle\sum}
}}\left\vert P(e_{j},\ldots,e_{j})\right\vert \leq C\left\Vert P\right\Vert
\text{.}%
\]
Applying the isometric isomorphism from $\mathcal{L}\left(  c_{0}%
;X_{k}\right)  $ onto $\ell_{1}^{w}(X_{k})$, for each $k=1,\ldots,m$, (see
\cite[Proposition 2.2]{diestel}) we are led to the final conclusion that every
$\left(  n_{1},\ldots n_{m}\right)  $-homogeneous polynomial $P:X_{1}%
\times\cdots\times X_{m}\rightarrow\mathbb{K}$ is absolutely $\left(
1;1\right)  $-summing.
\end{proof}

Recall that if $2\leq q\leq\infty$ and $\left(  r_{j}\right)  _{j=1}^{\infty}$
are the Rademacher functions, then $X$ has \textit{cotype }$q$ if there exists
$C_{q}\left(  X\right)  \geq0$ such that, for every $k\in\mathbb{N}$ and
$x_{1},\ldots,x_{k}\in X$,%
\[
\left(  \overset{k}{\underset{j=1}{%
{\displaystyle\sum}
}}\left\Vert x_{j}\right\Vert ^{q}\right)  ^{1/q}\leq C_{q}\left(  X\right)
\left(
{\displaystyle\int_{0}^{1}}
\left\Vert \overset{k}{\underset{j=1}{%
{\displaystyle\sum}
}}r_{j}\left(  t\right)  x_{j}\right\Vert ^{2}dt\right)  ^{1/2}\text{.}%
\]
To cover the case $q=\infty$, we replace $\left(
{\textstyle\sum_{j=1}^{k}}
\left\Vert x_{j}\right\Vert ^{q}\right)  ^{1/q}$ by $\max_{j\leq k}\left\Vert
x_{j}\right\Vert $. We denote $\inf\left\{  q;X\text{ has cotype }q\right\}  $
by $\cot X$.

The main connection between cotype and absolutely summing operators is given
by the following result:

\begin{theorem}
[Maurey--Talagrand \cite{talagr}]\label{mt}If $X$ has finite cotype $q$, then
the identity operator $id_{X}:X\rightarrow X$ is $\left(  q;1\right)
$-summing. The converse is true, except for $q=2$.
\end{theorem}

By exploiting the notion of cotype, we demonstrate other coincidence results.

\begin{proposition}
\label{coinc}Let $m\in\mathbb{N}$ and $\left(  n_{1},\ldots,n_{m}\right)
\in\mathbb{N}^{m}$.

\begin{description}
\item[(i)] If $X_{j}$ has cotype $q_{j}<\infty$ for each $j=1,\ldots,m$, then%
\[
\mathcal{P}_{\text{\textit{as}}\left(  s;1\right)  }\left(  ^{n_{1}}%
X_{1},\ldots,^{n_{m}}X_{m};Y\right)  =\mathcal{P}\left(  ^{n_{1}}X_{1}%
,\ldots,^{n_{m}}X_{m};Y\right)  \text{,}%
\]
for every $Y$ and every $s>0$ such that $\frac{1}{s}\leq\frac{n_{1}}{q_{1}%
}+\cdots+\frac{n_{m}}{q_{m}}$.

\item[(ii)] If $Y$ has cotype $q<\infty$, then%
\[
\mathcal{P}_{\text{\textit{as}}\left(  q;1\right)  }\left(  ^{n_{1}}%
X_{1},\ldots,^{n_{m}}X_{m};Y\right)  =\mathcal{P}\left(  ^{n_{1}}X_{1}%
,\ldots,^{n_{m}}X_{m};Y\right)  \text{,}%
\]
for every $X_{1},\ldots,X_{m}$.
\end{description}
\end{proposition}

\begin{proof}
(i) By \cite[Corollary 3.4]{Vel}, H\"{o}lder's Inequality and
Maurey--Talagrand's Theorem we have that%
\begin{align*}
\left(  \overset{n}{\underset{j=1}{%
{\displaystyle\sum}
}}\left\Vert P\left(  x_{j}^{\left(  1\right)  },\ldots,x_{j}^{\left(
m\right)  }\right)  \right\Vert ^{s}\right)  ^{\frac{1}{s}}  &  =\left\Vert
P\right\Vert \left(  \overset{n}{\underset{j=1}{%
{\displaystyle\sum}
}}\left\Vert x_{j}^{\left(  1\right)  }\right\Vert ^{n_{1}s}\cdots\left\Vert
\left(  x_{j}^{\left(  m\right)  }\right)  \right\Vert ^{n_{m}s}\right)
^{\frac{1}{s}}\\
&  \leq\left\Vert P\right\Vert \left(  \overset{n}{\underset{j=1}{%
{\displaystyle\sum}
}}\left\Vert x_{j}^{\left(  1\right)  }\right\Vert ^{q_{1}}\right)
^{\frac{n_{1}}{q_{1}}}\cdots\left(  \overset{n}{\underset{j=1}{%
{\displaystyle\sum}
}}\left\Vert x_{j}^{\left(  m\right)  }\right\Vert ^{q_{m}}\right)
^{\frac{n_{m}}{q_{m}}}\\
&  \leq\left\Vert P\right\Vert \left\Vert id_{X_{1}}\right\Vert _{as\left(
q_{1};1\right)  }^{n_{1}}\cdots\left\Vert id_{X_{m}}\right\Vert _{as\left(
q_{m};1\right)  }^{n_{m}}\overset{m}{\underset{k=1}{%
{\displaystyle\prod}
}}\left\Vert \left(  x_{j}^{\left(  k\right)  }\right)  _{j=1}^{n}\right\Vert
_{w,1}^{n_{k}}%
\end{align*}
for all $n\in\mathbb{N}$ and all $x_{1}^{\left(  k\right)  },\ldots
,x_{n}^{\left(  k\right)  }\in X_{k}$, with $k=1,\ldots,m$. Then, $P$ is
absolutely $\left(  s;1\right)  $-summing.

(ii) By Lemma \ref{dvmultipol}, it is straightforward that every
multipolynomial $P$ in \linebreak$\mathcal{P}\left(  ^{n_{1}}X_{1}%
,\ldots,^{n_{m}}X_{m};Y\right)  $ is such that%
\[
\left(  P\left(  x_{j}^{\left(  1\right)  },\ldots,x_{j}^{\left(  m\right)
}\right)  \right)  _{j=1}^{\infty}\in\ell_{1}^{w}\left(  Y\right)  \text{,}%
\]
whenever $\left(  x_{j}^{\left(  k\right)  }\right)  _{j=1}^{\infty}\in
\ell_{1}^{w}\left(  X_{k}\right)  $, $k=1,\ldots,m$. Besides, the Open Mapping
Theorem provides a constant $C\geq0$ such that%
\[
\left\Vert \left(  P\left(  x_{j}^{\left(  1\right)  },\ldots,x_{j}^{\left(
m\right)  }\right)  \right)  _{j=1}^{\infty}\right\Vert _{w,1}\leq C\left\Vert
P\right\Vert \overset{m}{\underset{k=1}{%
{\displaystyle\prod}
}}\left\Vert \left(  x_{j}^{\left(  k\right)  }\right)  _{j=1}^{\infty
}\right\Vert _{w,q_{k}}^{n_{k}}\text{,}%
\]
for all $\left(  x_{j}^{\left(  k\right)  }\right)  _{j=1}^{\infty}\in\ell
_{1}^{w}\left(  X_{k}\right)  $, $k=1,\ldots,m$. Therefore, it follows from
Maurey--Talagrand's Theorem that%
\begin{align*}
\left(  \overset{\infty}{\underset{j=1}{%
{\displaystyle\sum}
}}\left\Vert P\left(  x_{j}^{\left(  1\right)  },\ldots,x_{j}^{\left(
m\right)  }\right)  \right\Vert ^{q}\right)  ^{\frac{1}{q}}  &  =\left(
\overset{\infty}{\underset{j=1}{%
{\displaystyle\sum}
}}\left\Vert id_{Y}\left(  P\left(  x_{j}^{\left(  1\right)  },\ldots
,x_{j}^{\left(  m\right)  }\right)  \right)  \right\Vert ^{q}\right)
^{\frac{1}{q}}\\
&  \leq\left\Vert id_{Y}\right\Vert _{as\left(  q;1\right)  }\left\Vert
\left(  P\left(  x_{j}^{\left(  1\right)  },\ldots,x_{j}^{\left(  m\right)
}\right)  \right)  _{j=1}^{\infty}\right\Vert _{w,1}\\
&  \leq\left\Vert id_{Y}\right\Vert _{as\left(  q;1\right)  }C\left\Vert
P\right\Vert \overset{m}{\underset{k=1}{%
{\displaystyle\prod}
}}\left\Vert \left(  x_{j}^{\left(  k\right)  }\right)  _{j=1}^{\infty
}\right\Vert _{w,1}^{n_{k}}%
\end{align*}
for all $\left(  x_{j}^{\left(  k\right)  }\right)  _{j=1}^{\infty}\in\ell
_{1}^{w}\left(  X_{k}\right)  $, $k=1,\ldots,m$. We have shown that $P$ is
absolutely $\left(  q;1\right)  $-summing.
\end{proof}

In particular, we extract the coincidence results for the class of
polynomials/multilinear mappings due to Botelho \cite{boroyalirish}.

\begin{corollary}
[{\cite[Theorem 2.2]{boroyalirish}}]Let $m\in\mathbb{N}$.

\begin{description}
\item[(i)] If $X$ has cotype $mq<\infty$, then%
\[
\mathcal{P}_{\text{\textit{as}}\left(  q;1\right)  }\left(  ^{m}X;Y\right)
=\mathcal{P}\left(  ^{m}X;Y\right)  \text{, for every }Y\text{.}%
\]

\item[(ii)] If $Y$ has cotype $q$, then%
\[
\mathcal{P}_{\text{\textit{as}}\left(  q;1\right)  }\left(  ^{m}X;Y\right)
=\mathcal{P}\left(  ^{m}X;Y\right)  \text{, for every }X\text{.}%
\]

\end{description}
\end{corollary}

\begin{proof}
Apply Proposition \ref{coinc} with $m=1$.
\end{proof}

\begin{corollary}
[{\cite[Theorem 2.5]{boroyalirish}}]Let $m\in\mathbb{N}$.

\begin{description}
\item[(i)] If $X_{j}$ has cotype $q_{j}<\infty$ for each $j=1,\ldots,m$, then%
\[
\mathcal{L}_{\text{\textit{as}}\left(  s;1\right)  }\left(  X_{1},\ldots
,X_{m};Y\right)  =\mathcal{L}\left(  X_{1},\ldots,X_{m};Y\right)  \text{,}%
\]
for every $Y$ and every $s>0$ such that $\frac{1}{s}\leq\frac{1}{q_{1}}%
+\cdots+\frac{1}{q_{m}}$.

\item[(ii)] If $Y$ has cotype $q$, then%
\[
\mathcal{L}_{\text{\textit{as}}\left(  q;1\right)  }\left(  X_{1},\ldots
,X_{m};Y\right)  =\mathcal{L}\left(  X_{1},\ldots,X_{m};Y\right)  \text{,}%
\]
for every $X_{1},\ldots,X_{m}$.
\end{description}
\end{corollary}

\begin{proof}
Apply Proposition \ref{coinc} with $n_{1}=\cdots=n_{m}=1$.
\end{proof}

\section{Main result}

Let $m,n_{1},\ldots,n_{m}$ be natural numbers and let $X_{1},\ldots,X_{m}$ be
infinite dimensional Banach spaces with normalized unconditional Schauder
basis $\left(  x_{n}^{\left(  k\right)  }\right)  _{n\in\mathbb{N}}$,
$k=1,\ldots,m$. We define%
\[
\eta=\eta\left(  n_{1},\ldots,n_{m},X_{1},\ldots,X_{m},\left(  x_{n}^{\left(
1\right)  }\right)  _{n\in\mathbb{N}},\ldots,\left(  x_{n}^{\left(  m\right)
}\right)  _{n\in\mathbb{N}}\right)
\]
by%
\[
\eta=\inf\left\{  t;\left(  \overset{m}{\underset{k=1}{%
{\displaystyle\prod}
}}\left(  a_{j}^{\left(  k\right)  }\right)  ^{n_{k}}\right)  \in\ell
_{t}\text{ whenever }x_{k}=\underset{j=1}{\overset{\infty}{%
{\displaystyle\sum}
}}a_{j}^{\left(  k\right)  }x_{j}^{\left(  k\right)  }\in X_{k}\text{,
}k=1,\ldots,m\right\}
\]
and investigate the following general question:

\begin{problem}
\label{prob}If $\mathcal{P}_{\text{\textit{as}}\left(  q,1\right)  }\left(
^{n_{1}}X_{1},\ldots,^{n_{m}}X_{m};Y\right)  =\mathcal{P}\left(  ^{n_{1}}%
X_{1},\ldots,^{n_{m}}X_{m};Y\right)  $, then how is the behavior of $\eta$?
\end{problem}

The techniques to solve this kind of problem dates back to the seminal paper
\cite{lp} where J. Lindenstrauss and A. Pe\l czy\'{n}ski provide a beautiful
theorem stating that if $X$ is an infinite dimensional Banach space with an
unconditional Schauder basis and every linear operator from $X$ into an
infinite dimensional Banach space $Y$ is absolutely $\left(  1;1\right)
$-summing, then $X$ is isomorphic to $\ell_{1}\left(  \Gamma\right)  $ and $Y$
is isomorphic to a Hilbert space.

Firstly, recall that $Y$ \textit{finitely factors }the formal inclusion
$\ell_{p}\rightarrow\ell_{\infty}$ for $0<\delta<1$ if for every
$n\in\mathbb{N}$ there exist $y_{1},\ldots,y_{n}\in Y$ such that%
\[
\left(  1-\delta\right)  \left\Vert a\right\Vert _{\infty}\leq\left\Vert
\underset{j=1}{\overset{n}{%
{\displaystyle\sum}
}}a_{j}y_{j}\right\Vert \leq\left\Vert a\right\Vert _{p}\text{, for all
}a=\left(  a_{j}\right)  _{j=1}^{n}\in\ell_{p}^{n}\text{.}%
\]
Note that $1-\delta\leq\left\Vert y_{j}\right\Vert \leq1$, for all $j$.

For the extreme cases, some partial answers to the Problem \ref{prob} are
already known. More precisely, when $m=1$ and $n_{1}=1$, we regress to the
linear setting which as we comment has a solution in \cite{lp}. When $m=1$ and
$n_{1}=m>1$, D. Pellegrino \cite{pelcot} has been shown the following:

\begin{theorem}
[{\cite[Theorem 5]{pelcot}}]\label{thm5}Let $X$ and $Y$ be infinite
dimensional Banach spaces. Suppose that $X$ has an unconditional Schauder
basis. If $Y$ finitely factors the formal inclusion $\ell_{p}\rightarrow
\ell_{\infty}$ for some $\delta$ and $\mathcal{P}_{\text{\textit{as}}\left(
q;1\right)  }\left(  ^{m}X;Y\right)  =\mathcal{P}\left(  ^{m}X;Y\right)  $, then

\begin{description}
\item[(a)] $\eta\leq pq/\left(  p-q\right)  $, if $q<p$;

\item[(b)] $\eta\leq q$, if $q\leq p/2$.
\end{description}
\end{theorem}

\begin{theorem}
[{\cite[Theorem 5]{ppoloni}}]\label{thm5poloni}Let $X$ be an infinite
dimensional Banach space with an unconditional Schauder basis. If
$\mathcal{P}_{\text{\textit{as}}\left(  q;1\right)  }\left(  ^{m}X\right)
=\mathcal{P}\left(  ^{m}X\right)  $, then

\begin{description}
\item[(a)] $\eta\leq q/\left(  1-q\right)  $, if $q<1$;

\item[(b)] $\eta\leq q$, if $q\leq1/2$.
\end{description}
\end{theorem}

At another end, that is, when $m>1$ and $n_{1}=\cdots=n_{m}=1$, the following
has been proved:

\begin{theorem}
[{\cite[Theorem 8]{pelcot}}]\label{thm8}Let $Y$ be an infinite dimensional
Banach space and let $X_{1},\ldots,X_{m}$ denote infinite dimensional Banach
spaces with unconditional Schauder basis. If $Y$ finitely factors the formal
inclusion $\ell_{p}\rightarrow\ell_{\infty}$ for some $\delta$ and
$\mathcal{L}_{\text{\textit{as}}\left(  q;1\right)  }\left(  X_{1}%
,\ldots,X_{m};Y\right)  =\mathcal{L}\left(  X_{1},\ldots,X_{m};Y\right)  $, then

\begin{description}
\item[(a)] $\eta\leq pq/\left(  p-q\right)  $, if $q<p$;

\item[(b)] $\eta\leq q$, if $q\leq p/2$.
\end{description}
\end{theorem}

Later, still dealing with the class of homogeneous polynomials ($m=1$ and
$n_{1}=m>1$), G. Botelho and D. Pellegrino \cite{bp} obtained better estimates
for $\eta$, improving Theorems \ref{thm5} and \ref{thm5poloni}, as we see below:

\begin{lemma}
[Botelho--Pellegrino \cite{bp}]\label{jmaa}Suppose that $Y$ satisfies the
following condition:

There exist $C_{1},C_{2}>0$ and $p\geq1$ such that for every $n\in\mathbb{N}$,
there are $y_{1},\ldots,y_{n}$ in $Y$ with $\left\Vert y_{j}\right\Vert \geq
C_{1}$ for every $j$ and%
\[
\left\Vert \underset{j=1}{\overset{n}{%
{\displaystyle\sum}
}}a_{j}y_{j}\right\Vert \leq C_{2}\left(  \underset{j=1}{\overset{n}{%
{\displaystyle\sum}
}}\left\vert a_{j}\right\vert ^{p}\right)  ^{\frac{1}{p}}%
\]
for every $a_{1},\ldots,a_{n}\in\mathbb{K}$.

In this case, if $X$ has a normalized unconditional Schauder basis $\left(
x_{n}\right)  _{n\in\mathbb{N}}$, $q<p$ and $\mathcal{P}_{\text{\textit{as}%
}\left(  q;1\right)  }\left(  ^{m}X;Y\right)  =\mathcal{P}\left(
^{m}X;Y\right)  $, then $\eta\leq q$.
\end{lemma}

In this paper we will extend Lemma \ref{jmaa} to multipolynomials which will
provide a unified approach to Problem \ref{prob}. Indeed, next lemma recovers
all the aforementioned results as particular extreme cases, that is, one can
set $m=1$ to obtain the polynomial case (Lemma \ref{jmaa}) or $n_{1}%
=\cdots=n_{m}=1$ to yield the multilinear version of Lemma \ref{jmaa} (also an
up to now non-proved result) thus improving Theorem \ref{thm8}. Besides that
unifying contribution, the same result lifts to strict multipolynomials.

\begin{lemma}
\label{jmaamultipol}Suppose that $Y$ satisfies the following condition:

There exist $C_{1},C_{2}>0$ and $p\geq1$ such that for every $n\in\mathbb{N}$,
there are $y_{1},\ldots,y_{n}$ in $Y$ with $\left\Vert y_{j}\right\Vert \geq
C_{1}$ for every $j$ and%
\[
\left\Vert \underset{j=1}{\overset{n}{%
{\displaystyle\sum}
}}a_{j}y_{j}\right\Vert \leq C_{2}\left(  \underset{j=1}{\overset{n}{%
{\displaystyle\sum}
}}\left\vert a_{j}\right\vert ^{p}\right)  ^{\frac{1}{p}}%
\]
for every $a_{1},\ldots,a_{n}\in\mathbb{K}$.

In this case, if $X_{k}$ has a normalized unconditional Schauder basis
$\left(  x_{n}^{\left(  k\right)  }\right)  _{n\in\mathbb{N}}$, for each
$k=1,\ldots,m$, $q<p$ and $\mathcal{P}_{\text{\textit{as}}\left(  q;1\right)
}\left(  ^{n_{1}}X_{1},\ldots,^{n_{m}}X_{m};Y\right)  =\mathcal{P}\left(
^{n_{1}}X_{1},\ldots,^{n_{m}}X_{m};Y\right)  $, then $\eta\leq q$.
\end{lemma}

\begin{proof}
The proof follows the ideas of the original proof in \cite{bp}; it is done by
an induction argument. By the coincidence hypothesis, there exists $K>0$ such
that the absolutely summing multipolynomial norm $\pi_{\text{\textit{as}%
}\left(  q;1\right)  }^{\left(  n_{1},\ldots,n_{m}\right)  }\left(  P\right)
\leq K\left\Vert P\right\Vert $ for all $P\in$ $\mathcal{P}\left(  ^{n_{1}%
}X_{1},\ldots,^{n_{m}}X_{m};Y\right)  $. Let $n$ be a fixed natural number and
$\left\{  \mu_{j}\right\}  _{j=1}^{n}$ be such that $\sum_{j=1}^{n}\left\vert
\mu_{j}\right\vert ^{s}=1$ with $s=p/q$. Define $P:X_{1}\times\cdots\times
X_{m}\rightarrow Y$ by%
\[
P\left(  x_{1},\ldots,x_{m}\right)  =\underset{j=1}{\overset{n}{%
{\displaystyle\sum}
}}\left\vert \mu_{j}\right\vert ^{\frac{1}{q}}\left(  a_{j}^{(1)}\right)
^{n_{1}}\cdots\left(  a_{j}^{(m)}\right)  ^{n_{m}}y_{j}\text{,}%
\]
where $x_{k}=\sum_{j=1}^{\infty}a_{j}^{(k)}x_{j}^{\left(  k\right)  }$, for
$k=1,\ldots,m$. Since $\left(  x_{n}^{\left(  k\right)  }\right)
_{n\in\mathbb{N}}$ is an unconditional basis, there exist $\varrho_{k}>0$
satisfying%
\[
\left\Vert \overset{\infty}{\underset{j=1}{%
{\displaystyle\sum}
}}\varepsilon_{j}a_{j}^{(k)}x_{j}^{\left(  k\right)  }\right\Vert \leq
\varrho_{k}\left\Vert \underset{j=1}{\overset{\infty}{%
{\displaystyle\sum}
}}a_{j}^{(k)}x_{j}^{\left(  k\right)  }\right\Vert =\varrho_{k}\left\Vert
x_{k}\right\Vert \text{, for any }\varepsilon_{j}=\pm1\text{.}%
\]
Hence $\left\Vert \sum_{j=1}^{n}\varepsilon_{j}a_{j}^{(k)}x_{j}^{\left(
k\right)  }\right\Vert \leq\varrho_{k}\left\Vert x_{k}\right\Vert $ for all
$n$, any $\varepsilon_{j}=\pm1$ and all $k=1,\ldots,m$. So, if $x_{k}%
=\sum_{j=1}^{\infty}a_{j}^{(k)}x_{j}^{\left(  k\right)  }$, we have
$\left\vert a_{j}^{(k)}\right\vert \leq\varrho_{k}\left\Vert x_{k}\right\Vert
$ for all $j$, all $k=1,\ldots,m$ and then we get%
\begin{align*}
\left\Vert P\left(  x_{1},\ldots,x_{m}\right)  \right\Vert  &  =\left\Vert
\underset{j=1}{\overset{n}{%
{\displaystyle\sum}
}}\left\vert \mu_{j}\right\vert ^{\frac{1}{q}}\left(  a_{j}^{(1)}\right)
^{n_{1}}\cdots\left(  a_{j}^{(m)}\right)  ^{n_{m}}y_{j}\right\Vert \\
&  \leq C_{2}\left(  \underset{j=1}{\overset{n}{%
{\displaystyle\sum}
}}\left\vert \left\vert \mu_{j}\right\vert ^{\frac{1}{q}}\left(  a_{j}%
^{(1)}\right)  ^{n_{1}}\cdots\left(  a_{j}^{(m)}\right)  ^{n_{m}}\right\vert
^{p}\right)  ^{\frac{1}{p}}\\
&  \leq C_{2}\overset{m}{\underset{k=1}{%
{\displaystyle\prod}
}}\left(  \varrho_{k}\left\Vert x_{k}\right\Vert \right)  ^{n_{k}}\left(
\underset{j=1}{\overset{n}{%
{\displaystyle\sum}
}}\left\vert \mu_{j}\right\vert ^{s}\right)  ^{\frac{1}{p}}\\
&  =C_{2}\overset{m}{\underset{k=1}{%
{\displaystyle\prod}
}}\left(  \varrho_{k}\left\Vert x_{k}\right\Vert \right)  ^{n_{k}}\text{.}%
\end{align*}
We obtain $\left\Vert P\right\Vert \leq C_{2}%
{\textstyle\prod_{k=1}^{m}}
\varrho_{k}^{n_{k}}$ and $\pi_{\text{\textit{as}}\left(  q;1\right)
}^{\left(  n_{1},\ldots,n_{m}\right)  }\left(  P\right)  \leq KC_{2}%
{\textstyle\prod_{k=1}^{m}}
\varrho_{k}^{n_{k}}$ and achieve the estimate below:%
\begin{align}
\left(  \underset{j=1}{\overset{n}{%
{\displaystyle\sum}
}}\left\vert \left(  a_{j}^{(1)}\right)  ^{n_{1}}\cdots\left(  a_{j}%
^{(m)}\right)  ^{n_{m}}\left\vert \mu_{j}\right\vert ^{\frac{1}{q}}%
C_{1}\right\vert ^{q}\right)  ^{\frac{1}{q}} &  \leq\left(  \underset
{j=1}{\overset{n}{%
{\displaystyle\sum}
}}\left\Vert \left(  a_{j}^{(1)}\right)  ^{n_{1}}\cdots\left(  a_{j}%
^{(m)}\right)  ^{n_{m}}\left\vert \mu_{j}\right\vert ^{\frac{1}{q}}%
y_{j}\right\Vert ^{q}\right)  ^{\frac{1}{q}}\nonumber\\
&  =\left(  \underset{j=1}{\overset{n}{%
{\displaystyle\sum}
}}\left\Vert P\left(  a_{j}^{(1)}x_{j}^{\left(  1\right)  },\ldots,a_{j}%
^{(m)}x_{j}^{\left(  m\right)  }\right)  \right\Vert ^{q}\right)  ^{\frac
{1}{q}}\nonumber\\
&  \leq\pi_{\text{\textit{as}}\left(  q;1\right)  }^{\left(  n_{1}%
,\ldots,n_{m}\right)  }\left(  P\right)  \overset{m}{\underset{k=1}{%
{\displaystyle\prod}
}}\left\Vert \left(  a_{j}^{(k)}x_{j}^{\left(  k\right)  }\right)  _{j=1}%
^{n}\right\Vert _{w,1}^{n_{k}}\nonumber\\
&  \leq KC_{2}\overset{m}{\underset{k=1}{%
{\displaystyle\prod}
}}\varrho_{k}^{n_{k}}2^{n_{k}}\underset{\varepsilon_{j}\in\left\{
1,-1\right\}  }{\max}\left\{  \left\Vert \overset{n}{\underset{j=1}{%
{\textstyle\sum}
}}\varepsilon_{j}a_{j}^{(k)}x_{j}^{\left(  k\right)  }\right\Vert \right\}
^{n_{k}}\nonumber\\
&  \leq KC_{2}\overset{m}{\underset{k=1}{%
{\displaystyle\prod}
}}\left(  2\varrho_{k}^{2}\left\Vert x_{k}\right\Vert \right)  ^{n_{k}%
}\text{.}\label{jmaa1}%
\end{align}
Note that the last inequality holds whenever $\sum_{j=1}^{n}\left\vert \mu
_{j}\right\vert ^{s}=1$. Hence, since $1/s+1/\left(  \frac{s}{s-1}\right)
=1$, we have%
\begin{align}
&  \left(  \underset{j=1}{\overset{n}{%
{\displaystyle\sum}
}}\left\vert \left(  a_{j}^{(1)}\right)  ^{n_{1}}\cdots\left(  a_{j}%
^{(m)}\right)  ^{n_{m}}\right\vert ^{\frac{s}{s-1}q}\right)  ^{1/\left(
\frac{s}{s-1}\right)  }\nonumber\\
&  =\sup\left\{  \left\vert \underset{j=1}{\overset{n}{%
{\displaystyle\sum}
}}\mu_{j}\left(  a_{j}^{(1)}\right)  ^{n_{1}q}\cdots\left(  a_{j}%
^{(m)}\right)  ^{n_{m}q}\right\vert ;\underset{j=1}{\overset{n}{%
{\displaystyle\sum}
}}\left\vert \mu_{j}\right\vert ^{s}=1\right\}  \nonumber\\
&  \leq\sup\left\{  \underset{j=1}{\overset{n}{%
{\displaystyle\sum}
}}\left\vert \left(  a_{j}^{(1)}\right)  ^{n_{1}}\cdots\left(  a_{j}%
^{(m)}\right)  ^{n_{m}}\left\vert \mu_{j}\right\vert ^{\frac{1}{q}}\right\vert
^{q};\underset{j=1}{\overset{n}{%
{\displaystyle\sum}
}}\left\vert \mu_{j}\right\vert ^{s}=1\right\}  \text{.}\label{jmaa2}%
\end{align}
Then, by (\ref{jmaa1}) and (\ref{jmaa2}), it follows that%
\[
\left(  \underset{j=1}{\overset{n}{%
{\displaystyle\sum}
}}\left\vert \left(  a_{j}^{(1)}\right)  ^{n_{1}}\cdots\left(  a_{j}%
^{(m)}\right)  ^{n_{m}}\right\vert ^{\frac{s}{s-1}q}\right)  ^{1/\left(
\frac{s}{s-1}\right)  }\leq\left(  C_{1}^{-1}KC_{2}\overset{m}{\underset{k=1}{%
{\displaystyle\prod}
}}\left(  2\varrho_{k}^{2}\left\Vert x_{k}\right\Vert \right)  ^{n_{k}%
}\right)  ^{q}\text{,}%
\]
and then%
\[
\left(  \underset{j=1}{\overset{n}{%
{\displaystyle\sum}
}}\left\vert \left(  a_{j}^{(1)}\right)  ^{n_{1}}\cdots\left(  a_{j}%
^{(m)}\right)  ^{n_{m}}\right\vert ^{\frac{s}{s-1}q}\right)  ^{1/\left(
\frac{s}{s-1}\right)  q}\leq C_{1}^{-1}KC_{2}\overset{m}{\underset{k=1}{%
{\displaystyle\prod}
}}\left(  2\varrho_{k}^{2}\right)  ^{n_{k}}\overset{m}{\underset{k=1}{%
{\displaystyle\prod}
}}\left\Vert x_{k}\right\Vert ^{n_{k}}\text{.}%
\]
Since $\frac{s}{s-1}q=pq/\left(  p-q\right)  $ and $n$ is arbitrary, we have
$\eta\leq pq/\left(  p-q\right)  $. Now, if $q\leq p/2$, define, for a fixed
$n$, $S:X_{1}\times\cdots\times X_{m}\rightarrow Y$ by%
\[
S\left(  x_{1},\ldots,x_{m}\right)  =\underset{j=1}{\overset{n}{%
{\displaystyle\sum}
}}\left(  a_{j}^{(1)}\right)  ^{n_{1}}\cdots\left(  a_{j}^{(m)}\right)
^{n_{m}}y_{j}\text{, where }x_{k}=\overset{\infty}{\underset{j=1}{%
{\displaystyle\sum}
}}a_{j}^{(k)}x_{j}^{\left(  k\right)  }\text{, for }k=1,\ldots,m\text{.}%
\]
Since $p\geq\frac{s}{s-1}q$, combining the preceding estimates, we obtain%
\begin{align*}
\left\Vert S\left(  x_{1},\ldots,x_{m}\right)  \right\Vert  &  =\left\Vert
\underset{j=1}{\overset{n}{%
{\displaystyle\sum}
}}\left(  a_{j}^{(1)}\right)  ^{n_{1}}\cdots\left(  a_{j}^{(m)}\right)
^{n_{m}}y_{j}\right\Vert \\
&  \leq C_{2}\left(  \underset{j=1}{\overset{n}{%
{\displaystyle\sum}
}}\left\vert \left(  a_{j}^{(1)}\right)  ^{n_{1}}\cdots\left(  a_{j}%
^{(m)}\right)  ^{n_{m}}\right\vert ^{p}\right)  ^{\frac{1}{p}}\\
&  \leq C_{2}\left(  \underset{j=1}{\overset{n}{%
{\displaystyle\sum}
}}\left\vert \left(  a_{j}^{(1)}\right)  ^{n_{1}}\cdots\left(  a_{j}%
^{(m)}\right)  ^{n_{m}}\right\vert ^{\frac{s}{s-1}q}\right)  ^{1/\left(
\frac{s}{s-1}\right)  q}\\
&  \leq C_{1}^{-1}KC_{2}^{2}\overset{m}{\underset{k=1}{%
{\displaystyle\prod}
}}\left(  2\varrho_{k}^{2}\right)  ^{n_{k}}\overset{m}{\underset{k=1}{%
{\displaystyle\prod}
}}\left\Vert x_{k}\right\Vert ^{n_{k}}\text{.}%
\end{align*}
Thus, $\left\Vert S\right\Vert \leq C_{1}^{-1}KC_{2}^{2}%
{\textstyle\prod_{k=1}^{m}}
\left(  2\varrho_{k}^{2}\right)  ^{n_{k}}$ and $\pi_{\text{\textit{as}}\left(
q;1\right)  }^{\left(  n_{1},\ldots,n_{m}\right)  }\left(  S\right)  \leq
C_{1}^{-1}K^{2}C_{2}^{2}%
{\textstyle\prod_{k=1}^{m}}
\left(  2\varrho_{k}^{2}\right)  ^{n_{k}}$. Hence%
\begin{align*}
\underset{j=1}{\overset{n}{%
{\displaystyle\sum}
}}\left\vert \left(  a_{j}^{(1)}\right)  ^{n_{1}}\cdots\left(  a_{j}%
^{(m)}\right)  ^{n_{m}}C_{1}\right\vert ^{q} &  \leq\underset{j=1}{\overset
{n}{%
{\displaystyle\sum}
}}\left\Vert \left(  a_{j}^{(1)}\right)  ^{n_{1}}\cdots\left(  a_{j}%
^{(m)}\right)  ^{n_{m}}y_{j}\right\Vert ^{q}\\
&  =\underset{j=1}{\overset{n}{%
{\displaystyle\sum}
}}\left\Vert S\left(  a_{j}^{(1)}x_{j}^{\left(  1\right)  },\ldots,a_{j}%
^{(m)}x_{j}^{\left(  m\right)  }\right)  \right\Vert ^{q}\\
&  \leq\left(  \pi_{\text{\textit{as}}\left(  q;1\right)  }^{\left(
n_{1},\ldots,n_{m}\right)  }\left(  S\right)  \overset{m}{\underset{k=1}{%
{\displaystyle\prod}
}}2^{n_{k}}\underset{\varepsilon_{j}\in\left\{  1,-1\right\}  }{\max}\left\{
\left\Vert \overset{n}{\underset{j=1}{%
{\textstyle\sum}
}}\varepsilon_{j}a_{j}^{(k)}x_{j}^{\left(  k\right)  }\right\Vert \right\}
^{n_{k}}\right)  ^{q}\\
&  \leq\left(  C_{1}^{-1}K^{2}C_{2}^{2}\overset{m}{\underset{k=1}{%
{\displaystyle\prod}
}}\left(  4\varrho_{k}^{3}\right)  ^{n_{k}}\right)  ^{q}\overset{m}%
{\underset{k=1}{%
{\displaystyle\prod}
}}\left\Vert x_{k}\right\Vert ^{n_{k}q}\text{.}%
\end{align*}
Consequently, since $n$ is arbitrary, we have $%
{\textstyle\sum_{j=1}^{\infty}}
\left\vert \left(  a_{j}^{(1)}\right)  ^{n_{1}}\cdots\left(  a_{j}%
^{(m)}\right)  ^{n_{m}}\right\vert ^{q}<\infty$ whenever $x_{k}=%
{\textstyle\sum_{j=1}^{\infty}}
a_{j}^{(k)}x_{j}^{\left(  k\right)  }\in X_{k}$, for $k=1,\ldots,m$, and
$\eta\leq q$ if $q\leq p/2$.

Now we state the induction hypothesis:

Suppose that we have

\begin{description}
\item[(i)] $\eta\leq\frac{pq}{jp-jq}$ and $\left(
{\textstyle\sum_{i=1}^{\infty}}
\left\vert \left(  a_{i}^{(1)}\right)  ^{n_{1}}\cdots\left(  a_{i}%
^{(m)}\right)  ^{n_{m}}\right\vert ^{\frac{pq}{jp-jq}}\right)  ^{1/\left(
\frac{pq}{jp-jq}\right)  }\leq A_{j}%
{\textstyle\prod_{k=1}^{m}}
\left\Vert x_{k}\right\Vert ^{n_{k}}$, if $\frac{jp}{j+1}<q<p$,

\item[(ii)] $\eta\leq q$ and $\left(
{\textstyle\sum_{i=1}^{\infty}}
\left\vert \left(  a_{i}^{(1)}\right)  ^{n_{1}}\cdots\left(  a_{i}%
^{(m)}\right)  ^{n_{m}}\right\vert ^{q}\right)  ^{1/q}\leq B_{j}%
{\textstyle\prod_{k=1}^{m}}
\left\Vert x_{k}\right\Vert ^{n_{k}}$, if $q\leq\frac{jp}{j+1}$,
\end{description}

\noindent where

\begin{itemize}
\item $A_{1}=C_{1}^{-1}KC_{2}%
{\textstyle\prod_{k=1}^{m}}
\left(  2\varrho_{k}^{2}\right)  ^{n_{k}}$,

\item $B_{1}=C_{1}^{-2}K^{2}C_{2}^{2}%
{\textstyle\prod_{k=1}^{m}}
\left(  4\varrho_{k}^{3}\right)  ^{n_{k}}$,

\item $A_{j}=C_{1}^{-1}KC_{2}A_{j-1}%
{\textstyle\prod_{k=1}^{m}}
\left(  2\varrho_{k}\right)  ^{n_{k}}$ for $j\geq2$,

\item $B_{j}=C_{1}^{-2}K^{2}C_{2}^{2}A_{j-1}%
{\textstyle\prod_{k=1}^{m}}
\left(  4\varrho_{k}^{2}\right)  ^{n_{k}}$ for $j\geq2$.
\end{itemize}

Note that the case $j=1$ is done. We assume that (i) and (ii) hold for $j$ and
prove that they hold for $j+1$. To prove (i), assume $\frac{\left(
j+1\right)  p}{j+2}<q<p$.

Fix $n$ and let $\left\{  \mu_{i}\right\}  _{i=1}^{n}$ be such that $%
{\textstyle\sum_{i=1}^{n}}
\left\vert \mu_{i}\right\vert ^{s_{j}}=1$, where $s_{j}=\frac{p}{\left(
j+1\right)  q-jp}$. Defining $P$ as at the beginning and putting $l_{j}%
=\frac{pq}{jp-jq}$ and $t_{j}=\frac{pq}{\left(  j+1\right)  q-jp}$, we have
$\frac{1}{t_{j}}+\frac{1}{l_{j}}=\frac{1}{p}$ and so%
\begin{align*}
\left\Vert P\left(  x_{1},\ldots,x_{m}\right)  \right\Vert  &  =\left\Vert
\underset{i=1}{\overset{n}{%
{\displaystyle\sum}
}}\left\vert \mu_{i}\right\vert ^{\frac{1}{q}}\left(  a_{i}^{(1)}\right)
^{n_{1}}\cdots\left(  a_{i}^{(m)}\right)  ^{n_{m}}y_{i}\right\Vert \\
&  \leq C_{2}\left(  \underset{i=1}{\overset{n}{%
{\displaystyle\sum}
}}\left\vert \left\vert \mu_{i}\right\vert ^{\frac{1}{q}}\left(  a_{i}%
^{(1)}\right)  ^{n_{1}}\cdots\left(  a_{i}^{(m)}\right)  ^{n_{m}}\right\vert
^{p}\right)  ^{\frac{1}{p}}\\
&  \leq C_{2}\left(  \underset{i=1}{\overset{n}{%
{\displaystyle\sum}
}}\left\vert \left\vert \mu_{i}\right\vert ^{\frac{1}{q}}\right\vert ^{t_{j}%
}\right)  ^{\frac{1}{t_{j}}}\left(  \underset{i=1}{\overset{n}{%
{\displaystyle\sum}
}}\left\vert \left(  a_{i}^{(1)}\right)  ^{n_{1}}\cdots\left(  a_{i}%
^{(m)}\right)  ^{n_{m}}\right\vert ^{l_{j}}\right)  ^{\frac{1}{l_{j}}}\\
&  \leq C_{2}\left(  \underset{i=1}{\overset{n}{%
{\displaystyle\sum}
}}\left\vert \mu_{i}\right\vert ^{s_{j}}\right)  ^{\frac{1}{t_{j}}}\left(
\underset{i=1}{\overset{n}{%
{\displaystyle\sum}
}}\left\vert \left(  a_{i}^{(1)}\right)  ^{n_{1}}\cdots\left(  a_{i}%
^{(m)}\right)  ^{n_{m}}\right\vert ^{l_{j}}\right)  ^{\frac{1}{l_{j}}}\\
&  \leq C_{2}A_{j}\overset{m}{\underset{k=1}{%
{\displaystyle\prod}
}}\left\Vert x_{k}\right\Vert ^{n_{k}}\text{.}%
\end{align*}
We obtain $\left\Vert P\right\Vert \leq C_{2}A_{j}$ and $\pi
_{\text{\textit{as}}\left(  q;1\right)  }^{\left(  n_{1},\ldots,n_{m}\right)
}\left(  P\right)  \leq KC_{2}A_{j}$ and achieve the estimate below:%
\begin{align}
\left(  \underset{i=1}{\overset{n}{%
{\displaystyle\sum}
}}\left\vert \left(  a_{i}^{(1)}\right)  ^{n_{1}}\cdots\left(  a_{i}%
^{(m)}\right)  ^{n_{m}}\left\vert \mu_{i}\right\vert ^{\frac{1}{q}}%
C_{1}\right\vert ^{q}\right)  ^{\frac{1}{q}} &  \leq\left(  \underset
{i=1}{\overset{n}{%
{\displaystyle\sum}
}}\left\Vert \left(  a_{i}^{(1)}\right)  ^{n_{1}}\cdots\left(  a_{i}%
^{(m)}\right)  ^{n_{m}}\left\vert \mu_{i}\right\vert ^{\frac{1}{q}}%
y_{i}\right\Vert ^{q}\right)  ^{\frac{1}{q}}\nonumber\\
&  =\left(  \underset{i=1}{\overset{n}{%
{\displaystyle\sum}
}}\left\Vert P\left(  a_{i}^{(1)}x_{i}^{\left(  1\right)  },\ldots,a_{i}%
^{(m)}x_{i}^{\left(  m\right)  }\right)  \right\Vert ^{q}\right)  ^{\frac
{1}{q}}\nonumber\\
&  \leq\pi_{\text{\textit{as}}\left(  q;1\right)  }^{\left(  n_{1}%
,\ldots,n_{m}\right)  }\left(  P\right)  \overset{m}{\underset{k=1}{%
{\displaystyle\prod}
}}\left\Vert \left(  a_{i}^{(k)}x_{i}^{\left(  k\right)  }\right)  _{i=1}%
^{n}\right\Vert _{w,1}^{n_{k}}\nonumber\\
&  \leq KC_{2}A_{j}\overset{m}{\underset{k=1}{%
{\displaystyle\prod}
}}2^{n_{k}}\underset{\varepsilon_{i}\in\left\{  1,-1\right\}  }{\max}\left\{
\left\Vert \overset{n}{\underset{i=1}{%
{\textstyle\sum}
}}\varepsilon_{i}a_{i}^{(k)}x_{i}^{\left(  k\right)  }\right\Vert \right\}
^{n_{k}}\nonumber\\
&  \leq KC_{2}A_{j}\overset{m}{\underset{k=1}{%
{\displaystyle\prod}
}}\left(  2\varrho_{k}\left\Vert x_{k}\right\Vert \right)  ^{n_{k}}%
\text{.}\label{jmaa3}%
\end{align}
Since $\frac{1}{s_{j}}+1/\left(  \frac{s_{j}}{s_{j}-1}\right)  =1$, we have%
\begin{align}
&  \left(  \underset{i=1}{\overset{n}{%
{\displaystyle\sum}
}}\left\vert \left(  a_{i}^{(1)}\right)  ^{n_{1}}\cdots\left(  a_{i}%
^{(m)}\right)  ^{n_{m}}\right\vert ^{\frac{s_{j}}{s_{j}-1}q}\right)
^{1/\left(  \frac{s_{j}}{s_{j}-1}\right)  }\nonumber\\
&  =\sup\left\{  \left\vert \underset{i=1}{\overset{n}{%
{\displaystyle\sum}
}}\mu_{i}\left(  a_{i}^{(1)}\right)  ^{n_{1}q}\cdots\left(  a_{i}%
^{(m)}\right)  ^{n_{m}q}\right\vert ;\underset{i=1}{\overset{n}{%
{\displaystyle\sum}
}}\left\vert \mu_{i}\right\vert ^{s_{j}}=1\right\}  \nonumber\\
&  \leq\sup\left\{  \underset{i=1}{\overset{n}{%
{\displaystyle\sum}
}}\left\vert \left(  a_{i}^{(1)}\right)  ^{n_{1}}\cdots\left(  a_{i}%
^{(m)}\right)  ^{n_{m}}\left\vert \mu_{i}\right\vert ^{\frac{1}{q}}\right\vert
^{q};\underset{i=1}{\overset{n}{%
{\displaystyle\sum}
}}\left\vert \mu_{i}\right\vert ^{s_{j}}=1\right\}  \text{.}\label{jmaa4}%
\end{align}
It is plain that (\ref{jmaa3}) holds whenever $%
{\textstyle\sum_{i=1}^{n}}
\left\vert \mu_{i}\right\vert ^{s_{j}}=1$. Thus, by (\ref{jmaa3}) and
(\ref{jmaa4}), it follows that%
\[
\left(  \underset{i=1}{\overset{n}{%
{\displaystyle\sum}
}}\left\vert \left(  a_{i}^{(1)}\right)  ^{n_{1}}\cdots\left(  a_{i}%
^{(m)}\right)  ^{n_{m}}\right\vert ^{\frac{s_{j}}{s_{j}-1}q}\right)
^{1/\left(  \frac{s_{j}}{s_{j}-1}\right)  }\leq\left(  C_{1}^{-1}KC_{2}%
A_{j}\overset{m}{\underset{k=1}{%
{\displaystyle\prod}
}}\left(  2\varrho_{k}\left\Vert x_{k}\right\Vert \right)  ^{n_{k}}\right)
^{q}%
\]
and then%
\[
\left(  \underset{i=1}{\overset{n}{%
{\displaystyle\sum}
}}\left\vert \left(  a_{i}^{(1)}\right)  ^{n_{1}}\cdots\left(  a_{i}%
^{(m)}\right)  ^{n_{m}}\right\vert ^{\frac{s_{j}}{s_{j}-1}q}\right)
^{1/\left(  \frac{s_{j}}{s_{j}-1}\right)  q}\leq C_{1}^{-1}KC_{2}A_{j}%
\overset{m}{\underset{k=1}{%
{\displaystyle\prod}
}}\left(  2\varrho_{k}\right)  ^{n_{k}}\overset{m}{\underset{k=1}{%
{\displaystyle\prod}
}}\left\Vert x_{k}\right\Vert ^{n_{k}}\text{.}%
\]
Since $\frac{s_{j}}{s_{j}-1}q=\frac{pq}{\left(  j+1\right)  p-\left(
j+1\right)  q}$ and $n$ is arbitrary, we have $\eta\leq\frac{pq}{\left(
j+1\right)  p-\left(  j+1\right)  q}$ and%
\[
\left(  \underset{i=1}{\overset{\infty}{%
{\displaystyle\sum}
}}\left\vert \left(  a_{i}^{(1)}\right)  ^{n_{1}}\cdots\left(  a_{i}%
^{(m)}\right)  ^{n_{m}}\right\vert ^{\frac{pq}{\left(  j+1\right)  p-\left(
j+1\right)  q}}\right)  ^{1/\left(  \frac{pq}{\left(  j+1\right)  p-\left(
j+1\right)  q}\right)  }\leq A_{j+1}\overset{m}{\underset{k=1}{%
{\displaystyle\prod}
}}\left\Vert x_{k}\right\Vert ^{n_{k}}\text{,}%
\]
which proves (i) for $j+1$. To prove (ii), assume $q\leq\frac{\left(
j+1\right)  p}{j+2}$ and invoke, for a fixed $n$, $S$ again. We have
$\frac{pq}{\left(  j+1\right)  p-\left(  j+1\right)  q}\leq p$, so%
\begin{align*}
\left\Vert S\left(  x_{1},\ldots,x_{m}\right)  \right\Vert  &  =\left\Vert
\underset{i=1}{\overset{n}{%
{\displaystyle\sum}
}}\left(  a_{i}^{(1)}\right)  ^{n_{1}}\cdots\left(  a_{i}^{(m)}\right)
^{n_{m}}y_{i}\right\Vert \\
&  \leq C_{2}\left(  \underset{i=1}{\overset{n}{%
{\displaystyle\sum}
}}\left\vert \left(  a_{i}^{(1)}\right)  ^{n_{1}}\cdots\left(  a_{i}%
^{(m)}\right)  ^{n_{m}}\right\vert ^{p}\right)  ^{\frac{1}{p}}\\
&  \leq C_{2}\left(  \underset{i=1}{\overset{n}{%
{\displaystyle\sum}
}}\left\vert \left(  a_{i}^{(1)}\right)  ^{n_{1}}\cdots\left(  a_{i}%
^{(m)}\right)  ^{n_{m}}\right\vert ^{\frac{pq}{\left(  j+1\right)  p-\left(
j+1\right)  q}}\right)  ^{1/\left(  \frac{pq}{\left(  j+1\right)  p-\left(
j+1\right)  q}\right)  }\\
&  \leq C_{2}A_{j+1}\overset{m}{\underset{k=1}{%
{\displaystyle\prod}
}}\left\Vert x_{k}\right\Vert ^{n_{k}}\text{.}%
\end{align*}
Thus $\left\Vert S\right\Vert \leq C_{2}A_{j+1}$ and $\pi_{\text{\textit{as}%
}\left(  q;1\right)  }^{\left(  n_{1},\ldots,n_{m}\right)  }\left(  S\right)
\leq KC_{2}A_{j+1}$ and then we get%
\begin{align*}
\underset{i=1}{\overset{n}{%
{\displaystyle\sum}
}}\left\vert \left(  a_{i}^{(1)}\right)  ^{n_{1}}\cdots\left(  a_{i}%
^{(m)}\right)  ^{n_{m}}C_{1}\right\vert ^{q} &  \leq\underset{i=1}{\overset
{n}{%
{\displaystyle\sum}
}}\left\Vert \left(  a_{i}^{(1)}\right)  ^{n_{1}}\cdots\left(  a_{i}%
^{(m)}\right)  ^{n_{m}}y_{i}\right\Vert ^{q}\\
&  =\underset{i=1}{\overset{n}{%
{\displaystyle\sum}
}}\left\Vert S\left(  a_{i}^{(1)}x_{i}^{\left(  1\right)  },\ldots,a_{i}%
^{(m)}x_{i}^{\left(  m\right)  }\right)  \right\Vert ^{q}\\
&  \leq\left(  \pi_{\text{\textit{as}}\left(  q;1\right)  }^{\left(
n_{1},\ldots,n_{m}\right)  }\left(  S\right)  \underset{k=1}{%
{\displaystyle\prod}
}2^{n_{k}}\underset{\varepsilon_{i}\in\left\{  1,-1\right\}  }{\max}\left\{
\left\Vert \overset{n}{\underset{i=1}{%
{\textstyle\sum}
}}\varepsilon_{i}a_{i}^{(k)}x_{i}^{\left(  k\right)  }\right\Vert \right\}
^{n_{k}}\right)  ^{q}\\
&  \leq\left(  KC_{2}A_{j+1}\right)  ^{q}\overset{m}{\underset{k=1}{%
{\displaystyle\prod}
}}\left(  2\varrho_{k}\left\Vert x_{k}\right\Vert \right)  ^{n_{k}q}\text{.}%
\end{align*}
Consequently, since $n$ is arbitrary, we have%
\[
\left(  \underset{i=1}{\overset{\infty}{%
{\displaystyle\sum}
}}\left\vert \left(  a_{i}^{(1)}\right)  ^{n_{1}}\cdots\left(  a_{i}%
^{(m)}\right)  ^{n_{m}}\right\vert ^{q}\right)  ^{\frac{1}{q}}\leq
B_{j+1}\overset{m}{\underset{k=1}{%
{\displaystyle\prod}
}}\left\Vert x_{k}\right\Vert ^{n_{k}}%
\]
whenever $x_{k}=%
{\textstyle\sum_{i=1}^{\infty}}
a_{i}^{(k)}x_{i}^{\left(  k\right)  }\in X_{k}$, $k=1,\ldots,m$, proving (ii)
for $j+1$. The induction argument is done.

Finally, since $\lim_{j\rightarrow\infty}\frac{jp}{j+1}=p$, the proof is concluded.
\end{proof}

\begin{theorem}
\label{thm23multipol}Let $X_{1},\ldots,X_{m}$ be infinite dimensional Banach
spaces with normalized unconditional Schauder basis and $\mathcal{P}%
_{\text{\textit{as}}\left(  q;1\right)  }\left(  ^{n_{1}}X_{1},\ldots,^{n_{m}%
}X_{m};Y\right)  =\mathcal{P}\left(  ^{n_{1}}X_{1},\ldots,^{n_{m}}%
X_{m};Y\right)  $. Then $\eta\leq q$ if:

\begin{description}
\item[(i)] $q<1$ and $\dim Y<\infty$;

\item[(ii)] $q<\cot Y$ and $\dim Y=\infty$.
\end{description}

\begin{proof}
(i) It is immediate that we just need to consider $Y=\mathbb{K}$. Applying
Lemma \ref{jmaamultipol} with $p=C_{1}=C_{2}=y_{1}=\cdots=y_{n}=1$, the proof
is done.

(ii) Since Maurey--Pisier Theorem (see \cite[p. 226]{diestel}) asserts that
$Y$ finitely factors $\ell_{\cot Y}\hookrightarrow\ell_{\infty}$, it suffices
to call on Lemma \ref{jmaamultipol} with $p=\cot Y$.
\end{proof}
\end{theorem}

\begin{corollary}
[{\cite[Theorem 2.3]{bp}}]Let $X$ be an infinite dimensional Banach spaces
with a normalized unconditional Schauder basis and $\mathcal{P}%
_{\text{\textit{as}}\left(  q;1\right)  }\left(  ^{m}X;Y\right)
=\mathcal{P}\left(  ^{m}X;Y\right)  $. Then $\eta\leq q$ if:

\begin{description}
\item[(i)] $q<1$ and $\dim Y<\infty$;

\item[(ii)] $q<\cot Y$ and $\dim Y=\infty$.
\end{description}
\end{corollary}

\begin{proof}
Apply Theorem \ref{thm23multipol} with $m=1$.
\end{proof}

\begin{corollary}
Let $X_{1},\ldots,X_{m}$ be infinite dimensional Banach spaces with normalized
unconditional Schauder basis and $\mathcal{L}_{\text{\textit{as}}\left(
q;1\right)  }\left(  X_{1},\ldots,X_{m};Y\right)  =\mathcal{L}\left(
X_{1},\ldots,X_{m};Y\right)  $. Then $\eta\leq q$ if:

\begin{description}
\item[(i)] $q<1$ and $\dim Y<\infty$;

\item[(ii)] $q<\cot Y$ and $\dim Y=\infty$.
\end{description}
\end{corollary}

\begin{proof}
Apply Theorem \ref{thm23multipol} with $n_{1}=\cdots=n_{m}=1$.
\end{proof}

From now on it must be understood that we are considering the canonical basis.

\begin{corollary}
\label{cor22multipol}Let $m\in\mathbb{N}$ and $\left(  n_{1},\ldots
,n_{m}\right)  \in\mathbb{N}^{m}$.

\begin{description}
\item[(i)] If $1\leq r_{1},\ldots,r_{m}<\infty$ and $\dim Y=\infty$, we have%
\[
\mathcal{P}_{\text{\textit{as}}\left(  q;1\right)  }\left(  ^{n_{1}}%
\ell_{r_{1}},\ldots,^{n_{m}}\ell_{r_{m}};Y\right)  =\mathcal{P}\left(
^{n_{1}}\ell_{r_{1}},\ldots,^{n_{m}}\ell_{r_{m}};Y\right)
\]
$\Rightarrow$%
\[
q\geq\min\left\{  \frac{1}{\frac{n_{1}}{r_{1}}+\cdots+\frac{n_{m}}{r_{m}}%
},\cot Y\right\}  \text{.}%
\]

\item[(ii)] If $2\leq r_{1},\ldots,r_{m}<\infty$, $\dim Y=\infty$ and $Y$ has
cotype $\cot Y$, we have%
\[
\mathcal{P}_{\text{\textit{as}}\left(  q;1\right)  }\left(  ^{n_{1}}%
\ell_{r_{1}},\ldots,^{n_{m}}\ell_{r_{m}};Y\right)  =\mathcal{P}\left(
^{n_{1}}\ell_{r_{1}},\ldots,^{n_{m}}\ell_{r_{m}};Y\right)
\]
$\Leftrightarrow$%
\[
q\geq\min\left\{  \frac{1}{\frac{n_{1}}{r_{1}}+\cdots+\frac{n_{m}}{r_{m}}%
},\cot Y\right\}  \text{.}%
\]

\item[(iii)] For $2\leq r_{1},\ldots,r_{m}<\infty$, we have%
\[
\mathcal{P}_{\text{\textit{as}}\left(  q;1\right)  }\left(  ^{n_{1}}%
\ell_{r_{1}},\ldots,^{n_{m}}\ell_{r_{m}}\right)  =\mathcal{P}\left(  ^{n_{1}%
}\ell_{r_{1}},\ldots,^{n_{m}}\ell_{r_{m}}\right)
\]
$\Leftrightarrow$%
\[
q\geq\min\left\{  \frac{1}{\frac{n_{1}}{r_{1}}+\cdots+\frac{n_{m}}{r_{m}}%
},1\right\}  \text{.}%
\]

\end{description}
\end{corollary}

\begin{proof}
(i) If $q<\min\left\{  1/\left(  \frac{n_{1}}{r_{1}}+\cdots+\frac{n_{m}}%
{r_{m}}\right)  ,\cot Y\right\}  $, then Theorem \ref{thm23multipol} and
H\"{o}lder's Inequality provide $1/\left(  \frac{n_{1}}{r_{1}}+\cdots
+\frac{n_{m}}{r_{m}}\right)  =\eta\leq q$ (contradition).

(ii) Suppose that $q\geq\min\left\{  1/\left(  \frac{n_{1}}{r_{1}}%
+\cdots+\frac{n_{m}}{r_{m}}\right)  ,\cot Y\right\}  $. If $q\geq1/\left(
\frac{n_{1}}{r_{1}}+\cdots+\frac{n_{m}}{r_{m}}\right)  $, the result follows
from Proposition \ref{coinc}(i). If $q\geq\cot Y$, then it follows from
Proposition \ref{coinc}(ii). The converse follows from (i).

(iii) Assume the coincidence hypothesis and suppose that
\[
q<\min\left\{  1/\left(  \frac{n_{1}}{r_{1}}+\cdots+\frac{n_{m}}{r_{m}%
}\right)  ,1\right\}  \text{,}%
\]
then $\eta=1/\left(  \frac{n_{1}}{r_{1}}+\cdots+\frac{n_{m}}{r_{m}}\right)
>q$,\textbf{ }which contradicts Theorem \ref{thm23multipol}(i). Reciprocally,
if $q\geq1/\left(  \frac{n_{1}}{r_{1}}+\cdots+\frac{n_{m}}{r_{m}}\right)  $
apply Proposition \ref{coinc}(i) and we are done. If $q\geq1$, since $\ell
_{q}\supseteq\ell_{1}$, the proof is now a consequence of Lemma
\ref{dvmultipol}.
\end{proof}

We recover the original Botelho--Pellegrino's polynomial version.

\begin{corollary}
[{\cite[Corollary 2.2]{bp}}]Let $m\in\mathbb{N}$.

\begin{description}
\item[(i)] If $r\geq1$, $\dim Y=\infty$ and $Y$ has cotype $\cot Y$, we have%
\[
\mathcal{P}_{\text{\textit{as}}\left(  q;1\right)  }\left(  ^{m}\ell
_{r};Y\right)  =\mathcal{P}\left(  ^{m}\ell_{r};Y\right)  \Rightarrow
q\geq\min\left\{  \frac{r}{m},\cot Y\right\}  \text{.}%
\]

\item[(ii)] If $r\geq2$, $\dim Y=\infty$ and $Y$ has cotype $\cot Y$, we have%
\[
\mathcal{P}_{\text{\textit{as}}\left(  q;1\right)  }\left(  ^{m}\ell
_{r};Y\right)  =\mathcal{P}\left(  ^{m}\ell_{r};Y\right)  \Leftrightarrow
q\geq\min\left\{  \frac{r}{m},\cot Y\right\}  \text{.}%
\]

\item[(iii)] For $r\geq2$, we have $\mathcal{P}_{\text{\textit{as}}\left(
q;1\right)  }\left(  ^{m}\ell_{r}\right)  =\mathcal{P}\left(  ^{m}\ell
_{r}\right)  \Leftrightarrow q\geq\min\left\{  \frac{r}{m},1\right\}  $.
\end{description}
\end{corollary}

\begin{proof}
Apply Corollary \ref{cor22multipol} with $m=1$.
\end{proof}

Likewise, we naturally extract the multilinear version.

\begin{corollary}
Let $m\in\mathbb{N}$.

\begin{description}
\item[(i)] If $1\leq r_{1},\ldots,r_{m}<\infty$, $\dim Y=\infty$, we have%
\[
\mathcal{L}_{\text{\textit{as}}\left(  q;1\right)  }\left(  \ell_{r_{1}%
},\ldots,\ell_{r_{m}};Y\right)  =\mathcal{L}\left(  \ell_{r_{1}},\ldots
,\ell_{r_{m}};Y\right)
\]
$\Rightarrow$%
\[
q\geq\min\left\{  \frac{1}{\frac{1}{r_{1}}+\cdots+\frac{1}{r_{m}}},\cot
Y\right\}  \text{.}%
\]

\item[(ii)] If $2\leq r_{1},\ldots,r_{m}<\infty$, $\dim Y=\infty$ and $Y$ has
cotype $\cot Y$, we have%
\[
\mathcal{L}_{\text{\textit{as}}\left(  q;1\right)  }\left(  \ell_{r_{1}%
},\ldots,\ell_{r_{m}};Y\right)  =\mathcal{L}\left(  \ell_{r_{1}},\ldots
,\ell_{r_{m}};Y\right)
\]
$\Leftrightarrow$%
\[
q\geq\min\left\{  \frac{1}{\frac{1}{r_{1}}+\cdots+\frac{1}{r_{m}}},\cot
Y\right\}  \text{.}%
\]

\item[(iii)] For $2\leq r_{1},\ldots,r_{m}<\infty$, we have%
\[
\mathcal{L}_{\text{\textit{as}}\left(  q;1\right)  }\left(  \ell_{r_{1}%
},\ldots,\ell_{r_{m}}\right)  =\mathcal{L}\left(  \ell_{r_{1}},\ldots
,\ell_{r_{m}}\right)
\]
$\Leftrightarrow$%
\[
q\geq\min\left\{  \frac{1}{\frac{1}{r_{1}}+\cdots+\frac{1}{r_{m}}},1\right\}
\text{.}%
\]

\end{description}
\end{corollary}

\begin{proof}
Apply Corollary \ref{cor22multipol} $n_{1}=\cdots=n_{m}=1$.
\end{proof}

\end{document}